\documentclass[leqno,twoside]{article}

\usepackage{amsmath,amsthm,mathtools}
\usepackage{amssymb}
\usepackage{url}
\usepackage[T1]{fontenc}

\newtheorem{theorem}{Theorem}
\newtheorem{lemma}[theorem]{Lemma}
\newtheorem{corollary}[theorem]{Corollary}

 \theoremstyle{definition}

\newtheorem{remark}[theorem]{Remark}

\newcommand{\floor}[1]{\lfloor #1 \rfloor }

\begin{document}

\normalsize
\baselineskip=17pt

\title{\mbox{On the construction of absolutely normal numbers}}

\author{Christoph Aistleitner
\\Institute of Analysis and  Number Theory
\\Graz University of Technology
\\A-8010 Graz, Austria
\\E-mail:  aistleitner@math.tugraz.at
\and
Ver\'onica Becher
\\ Departamento de  Computaci\'on, Facultad de Ciencias Exactas y Naturales
\\Universidad de Buenos Aires \& ICC,  CONICET
\\Pabell\'on I, Ciudad Universitaria,  C1428EGA Buenos Aires, Argentina
\\E-mail: vbecher@dc.uba.ar
\and
Adrian-Maria Scheerer
\\Institute of Analysis and  Number Theory
\\Graz University of Technology
\\A-8010 Graz, Austria
\\E-mail: scheerer@math.tugraz.at
\and
Theodore A. Slaman
\\Department of Mathematics
\\University of California Berkeley
\\719 Evans Hall \#3840, Berkeley, CA 94720-3840 USA
\\E-mail: slaman@math.berkeley.edu
}

\date{July 8, 2017}

\maketitle

\renewcommand{\thefootnote}{}



\renewcommand{\thefootnote}{\arabic{footnote}}
\setcounter{footnote}{0}

\begin{abstract}
\noindent We give a construction of an absolutely normal real number~$x$ 
such that  for every integer $b $ greater than or equal to $2$, 
the discrepancy of the first $N$ terms 
of the sequence  $(b^n x \mod 1)_{n\geq 0}$ 
is of asymptotic order $\mathcal{O}(N^{-1/2})$.
This is below the order of discrepancy which holds for almost all real numbers. 
Even the existence of absolutely normal numbers having a discrepancy of 
such a small asymptotic order was not known before.
\end{abstract}

\section{Introduction and statement of results}

For a sequence  $(x_j)_{j\geq 0}$ of real numbers in the unit interval,
the discrepancy of the first~$N$ elements~is
\[
D_N((x_j)_{j\geq 0})= \sup_{0\leq \alpha_1 < \alpha_2 \leq 1} \left| \frac{1}{N} \#\{ j : 0 \leq j \leq N-1  \text{ and } \alpha_1 \leq x_j < \alpha_2 \} - (\alpha_2-\alpha_1)\ \right|.
\]
A sequence $(x_j)_{j\geq 0}$ of real numbers in the unit interval is uniformly distributed if and only if
$\lim_{N\to \infty} D_N((x_j)_{j\geq 0}) = 0$.

The property of Borel normality  can be defined in terms of uniform distribution.
For a real number $x$, we write $\{x\} = x-\floor{x}$ to denote the fractional part of $x$.
 A real number $x$  is normal with respect to an integer base $b$ greater than or equal to $2$
if the sequence $(\{b^j x\})_{j\geq 0}$ is uniformly distributed in the unit interval.
The numbers which are normal to all integer bases are called absolutely normal.
In this paper we prove the following theorem.

\begin{theorem}\label{th:main}
There is an  absolutely normal number~$x$  such that 
for each integer~$b \geq 2$, there are numbers  $N_0(b)$ and $C_b$
such that for all $N \geq N_0(b)$,
\begin{align*}
D_N( (\{b^j x\})_{j\geq 0}) \leq \frac{C_b}{ \sqrt{N}}.
\end{align*}
For the constant $C_b$ we can choose $C_b = 3433 \cdot b$.
Moreover, there is an algorithm that computes the first $N$ digits of the expansion of~$x$
in base $2$  after performing exponential in $N$  mathematical operations.
\end{theorem}

It follows from the  work of   
 G\'al and G\'al \cite{GalGal1964}
 that for almost all real numbers  (in the sense of Lebesgue measure)
 and for all integer bases $b$ greater than or equal to $2$
 the discrepancy of the sequence $(\{b^j x\})_{j\geq 0}$ obeys the law of iterated logarithm.
Philipp~\cite{Philipp1975}  gave explicit constants 
and Fukuyama~\cite[Corollary]{Fukuyama2008}) sharpened the result.
He proved that for every real  $\theta>1$  there is a constant $C_\theta$ such that 
for almost all real $x$ we have
\[
\limsup_{N \to \infty} \frac{ \sqrt{N} D_N( (\{\theta^j x\})_{j\geq 0} )}{\sqrt{\log \log N}} = C_\theta.
\]
In case~$\theta$ is an integer greater than or equal to $2$, for $C_\theta$ one has the values
\[
C_\theta=\left\{
\begin{array}{ll}\medskip
\sqrt{84}/9, & \text{ if } \theta=2,
\\\medskip
\sqrt{(\theta+1)/(\theta-1)}/\sqrt{2}, &  \text{ if  $\theta$ is odd,}
\\\medskip
\sqrt{(\theta+1)\theta(\theta-2)/(\theta-1)^3}/\sqrt{2}, & \text{ if } \theta\geq 4 \text{ is even}.  
\end{array}
\right.
\]

To prove Theorem~\ref{th:main} we give a construction of a real number~$x$ 
such that, for every integer $b$ greater than or equal to $2$,
$D_N( (\{b^j x\})_{j\geq 0} )$ is of asymptotic order $\mathcal{O}(N^{-1/2})$,
hence,  below the order of discrepancy that holds for almost all real numbers. 
The existence of absolutely normal numbers having a discrepancy of 
such a small asymptotic order was not known before.

To prove Theorem \ref{th:main}, we define a computable sequence of nested binary intervals $(\Omega_k)_{k\geq 1}$
such that for all elements of $\Omega_k$ the discrepancy $D_N( (\{b^j x\})_{j\geq 0} )$ is sufficiently small for some range of $b$ and $N$. This argument uses methods going back to  G\'al and G\'al~\cite{GalGal1964} and Philipp~\cite{Philipp1975}.
The unique point in the intersection $\bigcap_{k\geq 1} \Omega_k$ is a computable number which satisfies the discrepancy estimate in the conclusion of the theorem.  This is the number we  obtain. The construction  uses just discrete mathematics and 
 yields directly the binary expansion of the computed  number.
Unfortunately, the algorithm that  computes the first~$N$ digits 
performs exponential in $N$ many operations.

In view of the method used to prove Theorem \ref{th:main}, the appearance of a bound of square-root order for the discrepancy is very natural. Note that the discrepancy is exactly the same as the Kolmogorov--Smironov statistic, applied to the case of the uniform distribution on $[0,1]$. By Kolmogorov's limit theorem the Kolmogorov--Smirnov statistic of a system of independent, identically distributed (i.i.d.) random variables has a limit distribution when normalized by $\sqrt{N}$ (somewhat similar to the case of the central limit theorem). Since it is well-known that so-called lacunary function systems (such as the system $(\{b^j x\})_{j \geq 0}$ for $b \geq 2$) exhibit properties which are very similar to those of independent random systems, we can expect a similar behavior for the discrepancy of $(\{b^j x\})_{j \geq 0}$. In other words, we can find a set of values of $x$ which has positive measure, and whose discrepancy is below some appropriate constant times the square-root normalizing factor (see \cite{A-B-distr} for more details). Since the sequence $(b^j)_{j \geq 0}$ is very quickly increasing, we can iterate this argument and find a ``good'' set of values of $x$ (which we call $\Omega_k$) which gives the desired discrepancy bound and which has positive measure \emph{within} the previously constructed set $\Omega_{k-1}$. These remarks show why a discrepancy bound of order $N^{-1/2}$ is a kind of barrier when constructing the absolutely normal number $x$ using probabilistic methods. Accordingly, any further improvement of Theorem \ref{th:main} would require some truly novel ideas.


As reported in ~\cite{scheerer}, prior to the present work 
the construction of an absolutely normal number with the smallest 
discrepancy bound  was due to Levin~\cite{Levin1979}.
Given a countable set~$L$ of  reals greater than~$1$, 
Levin constructs a real number $x$ such that 
for every~$\theta$ in~$L$,
\[
D_N((\{\theta^j x\})_{j\geq 0}) <  \frac{C_\theta (\log N)^3}{\sqrt{N}},
\]
for a constant   $C_\theta$ for every $N\geq N_0(\theta)$.
His construction does not produce directly the binary expansion of the defined number $x$.
Instead it  produces a computable sequence of real 
numbers that converge to~$x$  
and the computation of the $N$-th term requires  double-exponential (in $N$) many operations including
 trigonometric operations, see~\cite{AlvarezBecher2016}. 

It  is possible to prove a version of Theorem~\ref{th:main} 
replacing the set of integer bases  by any subset of  computable 
 reals greater than~$1$.
The proof would remain  essentially the same except for a suitable version of  
 Lemma~\ref{lemma:bound}.
In contrast, we do not know if it is possible obtain a version of Theorem~\ref{th:main} 
where the exponential computational complexity is replaced with polynomial computational complexity 
as in~\cite{Becher&Heiber&Slaman:2013}.  

 Theorem~\ref{th:main} does not supersede the discrepancy bound obtained by Levin~\cite{Levin:1999} for the discrepancy of a normal number with respect to one fixed base. For a fixed integer $b\geq 2$, Levin constructed a real number $x$ such that  
\[
D_N(\{b^j x\}_{j\geq 0})<  \frac{C_\theta (\log N)^2}{N}.
\]
One should compare this upper bound with the lower bound 
obtained by  Schmidt \cite{Schmidt:1972}, who 
proved that there is a constant $C$ for {\em every} 
sequence~$(x_j)_{j\geq 0}$ of real numbers in the unit interval
there are infinitely many $N$s such that 
\[
D_N((x_j)_{j\geq 0}) >C \frac{\log N}{ N}.
\]
This lower bound is achieved by some so-called low-discrepancy sequences, 
(see~\cite{Bugeaud:2012} and the references there), but it remains an important open problem whether this optimal order of discrepancy can also be achieved by a sequence of the form~$(\{b^j x\})_{j\geq 0}$ for a real  number~$x$.

Accordingly, two central questions in this field remain open:
\begin{itemize}
\item Asked by Korobov \cite{Korobov:1955}: 
For a \emph{fixed} integer~$b\geq 2$, 
what is the function~$\psi(N)$ with maximal speed of decrease to zero such that 
there is a real number $x$ for which
$$
D_N(\{b^j x\}_{j\geq 0})=  \mathcal{O} \left(\psi(N)\right) \qquad \text{as $N \to \infty$?}
$$

\item Asked by Bugeaud (personal communication, 2017): 
Is there a number $x$ satisfying the minimal discrepancy estimate for normality not only in one fixed base,
but in all bases at the same time? More precisely, let $\psi$ be Korobov's function from above. Is there a real number $x$ such that for \emph{all} integer bases $b\geq 2$, 
\[
D_N(\{b^j x\}_{j\geq 0}) = \mathcal{O} \left( \psi(N) \right) \qquad \text{as $N \to \infty$?}
\]
\end{itemize}

\section{Definitions and lemmas} \label{sec:1}

We use some tools from~\cite{GalGal1964,Philipp1975}. 
For non-negative integers $M$ and $N$, 
for a sequence of real numbers $(x_j)_{j\geq 0}$
and  for real numbers $\alpha_1, \alpha_2$ 
such that $0\leq \alpha_1<\alpha_2\leq 1$, we define
\begin{align*}
F\left(M,N,\alpha_1, \alpha_2, (x_j)_{j\geq 1}\right) =
\Big|& \#\{j: M\leq  j < M+N:  \alpha_1\leq x_j< \alpha_2 \} -  (\alpha_2-\alpha_1) N \Big|.
\end{align*}
To shorten notations we will write  $\{b^j x\}_{j\geq 0}$ to denote
$(\{b^j x\})_{j\geq 0} $.
Throughout the paper we will use the fact that
\[
F\left(M,N, \alpha_1, \alpha_2, \{b^j x\}_{j\geq 0}\right) = F\left(0,N, \alpha_1, \alpha_2, \{b^{j+M} x\}_{j\geq 0}\right)
\]
for every non-negative integer $M$.

The following lemma is a classical result from probability theory called 
Bernstein's inequality (see for example~\cite[Lemma 2.2.9]{sw}).
We write $\mu$ for the Lebesgue measure and occasionally we write $\exp(x)$ for $e^x$.
\begin{lemma} \label{lemma:bern}
Let $X_1, \dots, X_n$ be i.i.d. random variables having zero mean and variance $\sigma^2$, and assume that their absolute value is at most $1$. Then for every $\varepsilon > 0$
\[
\mathbb{P} \left( \left| \sum_{k=1}^n X_k \right| > \varepsilon \sqrt{n} \right) \leq 2 \exp\left( \frac{-\varepsilon^2}{2\sigma^2 + 2/3 \varepsilon n^{-1/2}} \right).
\]
\end{lemma}

\begin{lemma} \label{lemma:bound}
Let $b \geq 2$ be an integer, let $h$ and $N$ be positive integers such that $N \geq h$, and let  $\varepsilon$ be a positive real. Then for all integers  $M \geq 0$ and $a$ satisfying $0\leq a<b^h$,
\[
\mu \left(\left\{x\in (0,1): 
F \left(M,N,a b^{-h},(a+1) b^{-h} ,\{b^j x\}_{j\geq 0} \right) > \varepsilon \sqrt{h N} \right\} \right)
\]
is at most
\[
2 h \exp\left( \frac{-\varepsilon^2}{2b^{-h} (1-b^{-h}) + 2/3 \varepsilon \lfloor N / h \rfloor^{-1/2}} \right).
\]
\end{lemma}

\begin{proof}[Proof of Lemma~\ref{lemma:bound}]
We split the index set $\{M, M+1, \dots, M+N-1\}$ into $h$ classes, 
according to the remainder of an index when it is reduced modulo~$h$. 
Then each of these classes contains either $\lfloor N/h \rfloor$ or 
$\lceil N/h \rceil$ elements. Let $\mathbf{1}_{[a b^{-h},(a+1) b^{-h})} (x)$ 
denote the indicator function of the interval $[a b^{-h},(a+1) b^{-h})$. 
Let $\mathcal{M}_0$ denote the class of all indices in $\{M, \dots, M+N-1\}$ 
which leave remainder zero when being reduced modulo~$h$. 
Set $n_0 = \# \mathcal{M}_0$. Then it is an easy exercise to check that 
the system of functions 
$\big(\mathbf{1}_{[a b^{-h},(a+1) b^{-h})} (\{b^j x\}) - b^{-h}\big)_{j \in \mathcal{M}_0}$ 
is a system of i.i.d.\ random variables over the unit interval, 
equipped with Borel sets and Lebesgue measure.\footnote{These functions 
are Rademacher functions, just in base $b^h$ instead of the usual base $2$.
 See for example~\cite[Section 1.1.3]{strook} for more details.} 
The absolute value of these random variables is trivially bounded by $1$, 
they have mean zero, and their variance is $b^{-h} (1-b^{-h})$. 
Thus by Lemma~\ref{lemma:bern} we have
\begin{eqnarray*}
& & \mu \left( \left\{x \in (0,1): \left|\sum_{j \in \mathcal{M}_0} \left( \mathbf{1}_{[a b^{-h},(a+1) b^{-h})} (\{b^j x\}) - b^{-h} \right) \right| > \varepsilon \sqrt{n_0} \right\} \right) \\
& \leq & 2 \exp\left( \frac{-\varepsilon^2}{2b^{-h} (1-b^{-h}) + 2/3 \varepsilon n_0^{-1/2}} \right).
\end{eqnarray*}
Clearly similar estimates hold for the indices in the other residue classes. Let $n_1, \dots, n_{h-1}$ denote the cardinalities of these other residue classes. By assumption $n_0 + \dots + n_{h-1} = N$.  Note that by the Cauchy-Schwarz inequality we have $\sqrt{n_0} + \dots + \sqrt{n_{h-1}} \leq \sqrt{h}\sqrt{N}$. Thus, summing up, we obtain
\begin{eqnarray*}
& & \mu \left( \left\{ x \in (0,1): \left|\left(\sum_{j=M}^{M+N-1} \mathbf{1}_{[a b^{-h},(a+1) b^{-h})} (\{b^j x\})\right) - N b^{-h} \right| > \varepsilon \sqrt{h N} \right\} \right) \\
& \leq & 2 h \exp\left( \frac{-\varepsilon^2}{2b^{-h} (1-b^{-h}) + 2/3 \varepsilon \lfloor N / h \rfloor^{-1/2}} \right).
\end{eqnarray*}
This proves the lemma.
\end{proof}

We will use a modified version of Lemma~\ref{lemma:bound}, which works on any subinterval~$A$ of~$[0,1]$. 

\begin{lemma}\label{lemma:bound:mod}
Let $b \geq 2$ be an integer, let $h$ and $N$ be positive integers such that $N \geq h$, and let  $\varepsilon$ be a positive real. Then for all integers  $M \geq 0$ and $a$ satisfying $0\leq a<b^h$, for any subinterval $A$ of $[0,1]$ and for any positive integer $j_0$,
\[
\mu \left(\left\{x\in A: 
F(M+j_0,N,a b^{-h},(a+1) b^{-h} ,\{b^j x\}_{j\geq 0})  > \varepsilon \sqrt{h N} \right\} \right)
\]
is at most
\[
2 \mu(A) h\exp\left( \frac{-\varepsilon^2}{2b^{-h} (1-b^{-h}) + 2/3 \varepsilon \lfloor N / h \rfloor^{-1/2}} \right) + 2b^{-j_0}.
\]
\end{lemma}

\begin{proof}[Proof of Lemma~\ref{lemma:bound:mod}]
Let $B$ denote the largest interval contained in $A$ which has 
the property that both of its endpoints are integer multiples of 
$b^{-j_0}$. Then $\mu(A \backslash B) \leq 2b^{-j_0}$. 
Furthermore, by periodicity we have
\begin{eqnarray*}
&& \mu \left(\{x\in B: F \left(M + j_0,N, a b^{-m}, (a+1) b^{-m}, \{b^{j} x\}_{j\geq 0} \right) >\varepsilon  \sqrt{hN} \} \right) \\
& = & \mu(B) \cdot \mu \left(\{x\in (0,1): F \left(M,N, a b^{-m}, (a+1) b^{-m}, \{b^{j} x\}_{j\geq 0} \right) >\varepsilon  \sqrt{hN} \}\right),
\end{eqnarray*}
for which we can apply the conclusion of Lemma~\ref{lemma:bound}. 
Note that $\mu(B) \leq \mu(A)$. This proves Lemma~\ref{lemma:bound:mod}.
\end{proof}
The following corollary follows easily from Lemma~\ref{lemma:bound:mod}.

\begin{corollary} \label{lemma:bound:cor}
Let $b \geq 2$ be an integer, let $h$ and $N$ be positive integers such 
that $N \geq h$, and assume that $\varepsilon$ satisfies
\begin{eqnarray} \label{eps:criterion}
2/3 \varepsilon  \lfloor N / h \rfloor^{-1/2} \leq \frac{1}{b h^5}.
\end{eqnarray}
Then for all integers  $M \geq 0$ and $a$ satisfying 
$0\leq a<b^h$, for any subinterval $A$ of $[0,1]$ 
and for any positive integer $j_0$ we have
\[
\mu \left(\left\{x\in A: 
F \left(M+j_0,N,a b^{-h},(a+1) b^{-h} ,\{b^j x\}_{j\geq 0} \right)  > \varepsilon \sqrt{h N} \right\} \right)
\]
is at most
\[
\mu(A) 2 h \exp\left( \frac{-\varepsilon^2 b h^5}{529} \right) + 2b^{-j_0}.
\]
\end{corollary}

\begin{proof}
The corollary follows from Lemma~\ref{lemma:bound:mod} and the fact that 
\[
2b^{-h} (1-b^{-h}) \leq 2 b^{-h} \leq 528 b^{-1} h^{-5}
\] 
for all $b \geq 2$ and $h \geq 1$ 
(for the second inequality in the displayed formula it 
is sufficient to check that $2^{-h+2} \leq 528 h^{-5}$ for integers 
$h \geq 1$, which can be done numerically). 
Together with assumption \eqref{eps:criterion} this implies 
that $2b^{-h} (1-b^{-h}) + 2/3 \varepsilon \lfloor N / h \rfloor^{-1/2} \leq 529 b^{-1} h^{-5}$.
\end{proof}

\begin{remark}\label{rk:basic}
For any two reals $\alpha_1, \alpha_2$ such that $0\leq \alpha_1<\alpha_2<1$, 
and for any sequence $(x_j)_{j\geq 1}$ of reals,
a  trivial bound yields 
\[
F(0,N,\alpha_1,\alpha_2,(x_j)_{j\geq 1}) \leq 2 \sup_{\alpha \in [0,1)} F(0,N,0,\alpha,(x_j)_{j\geq 1}).
\]
And,   for any real number $\alpha \in (0,1)$, for any sequence of real numbers $(x_j)_{j\geq 1}$,
and for any non-negative  integers $N$ and $k$ we have
\[
F(0,N,0,\alpha,(x_j)_{j\geq 1}) \leq N/b^{k} +
\sum_{h=1}^k (b-1) \max_{ 0\leq a < b^h} F(0,N,  a b^{-h},  (a+1) b^{-h},(x_j)_{j\geq 1}).
\]
This observation follows from the fact that every interval $[0,\alpha)$ 
can be covered by at most $(b-1)$ intervals of length $b^{-1}$, 
at most $(b-1)$ intervals of length $b^{-2}$, and so on, 
at most $(b-1)$ intervals of length $b^{-k}$, and finally 
one additional interval of length $b^{-k}$. 
This decomposition can be easily derived from the digital 
representation of $\alpha$ in base $b$. 
\end{remark}
%
%
%
The index set can be decomposed in intervals between powers of $2$, 
and every possible initial segment of the index set can be written as 
a disjoint union of such  sets. This fact is expressed in the following lemma.

\begin{lemma}[adapted from \protect{\cite[Lemma 4]{Philipp1975}}] \label{lemma:philipp}
Let $b \geq 2$ be an integer,
let $N$ be a positive integer and let $n$ be such that 
$2^{n-1} < N \leq 2^{n}$, and let $M$ be a non-negative integer.
Then, there are non-negative integers $m_1,\ldots, m_n$  
such that $m_\ell 2^\ell + 2^{\ell-1} \leq N$ for  $\ell=1, \ldots, n$,
and such that for any positive integer $h$ and any $a$,
with $0\leq a<b^h$,
\begin{eqnarray*}
& & F(M,N,a b^{-h}, (a+1) b^{-h},\{b^j x\}_{j\geq 0}) \\
& \leq & N^{1/2} + \sum_{n/2 \leq \ell \leq n} F(M+ m_\ell 2^\ell, 2^{\ell-1},a b^{-h}, (a+1) b^{-h}, \{b^j x\}_{j\geq 0}).
\end{eqnarray*}
\end{lemma}

For the proof of Theorem~\ref{th:main} we proceed by induction, and define a sequence of nested binary
intervals $(\Omega_k)_{k \geq 1}$ which gives us the 
binary digits of the absolutely normal 
number which we want to construct. 
Set $\Omega_1 = \Omega_2 = \dots = \Omega_{99} = (0,1)$ 
for the start of the induction. (We start the induction at $k=100$ 
in order to avoid trivial notational problems with small values of~$k$.) 
We will always assume that $b \leq k$, so in step $k$ only bases~$b$ 
from~$2$ up to~$k$ are considered. 
Different bases are added gradually as the induction steps forward.

For integers $k \geq 100$ and $b$ such that $2 \leq b \leq k$ we set 
\[
N_k^{(b)} =  \left\lceil 2^{k} \frac{\log 2}{\log b} \right\rceil.
\]
We define sets
\[
\mathcal{N}_k^{b} = \left\{N \in \mathbb{N}:~N_k^{(b)} + 4k < N \leq N_{k+1}^{(b)} \right\}, \qquad k \geq 100, ~2 \leq b \leq k,
\]
and
\[
\mathcal{R}_k^{b} = \left\{N \in \mathbb{N}:~ N_k^{(b)} < N \leq N_k^{(b)} + 4k \right\}, \qquad k \geq 100, ~2 \leq b \leq k.
\]
The indices in $\bigcup_k \mathcal{R}_k^b$ are the ``remainder'', and do not give a relevant contribution. 
Their purpose is to separate the elements of $\mathcal{N}_k^b$ 
from those of $\mathcal{N}_{k+1}^b$, so that $b^{j_2}$ is significantly 
larger than $b^{j_1}$ whenever $j_2 \in \mathcal{N}_{k+1}^b$ and $j_1 \in \mathcal{N}_k^{(b)}$. 
The sets $\mathcal{N}_k^b$ and~$\mathcal{R}_k^b$ are constructed in 
such a way that they form a partition of~$\mathbb{N}$, 
except for finitely many initial elements of~$\mathbb{N}$. 
Precisely, one can check that these sets form a 
partition of~$\mathbb{N} \backslash \left\{1, \dots, N_{\max\{100,b\}}^b \right\}$.

For the induction step, assume that $k \geq 100$ and that the interval $\Omega_{k-1}$ is already defined, and that the length of $\Omega_{k-1}$ is bounded below by
\begin{equation} \label{omega_measure}
\mu(\Omega_{k-1}) \geq 2^{-2^{k} - k}.
\end{equation}
Set
\[
n_{k}^{(b)} = \left\lceil \log_2 \big(N_{k+1}^{(b)} - N_k^{(b)} - 4k\big) \right\rceil.
\]
and
\[
T_b(k)=\left\lceil \frac{n_k^{(b)}  \log 2}{2 \log b} \right\rceil.
\]
For non-negative integers $b, a, h, \ell$ such that 
\begin{align*}
2 \leq b \leq k, \quad  0\leq a< b^h, \quad
1\leq h\leq T_b(k), \quad n_k^{(b)}/2\leq \ell\leq n_k^{(b)},
\end{align*}
and non-negative integers $m_\ell$ such that
\begin{equation} \label{mell}
N_k^{(b)} + 4k + m_\ell 2^\ell + 2^{\ell-1} \leq N_{k+1}^{(b)}
\end{equation}
we define the sets 
\begin{eqnarray*}
 H(b,k,a, h, \ell, m_\ell) 
& = & \left\{x\in \Omega_{k-1}: F \left(N_k^{(b)} + 4k + m_\ell 2^{\ell},2^{\ell-1},a b^{-h},(a+1) b^{-h},\{b^j x\}_{j\geq 0} \right) \right. \\
& & \qquad \qquad \qquad\qquad\qquad \qquad \left. > 46 \cdot 2^{(\ell-1)/2} h^{-3/2} (n_k^{(b)}-\ell +1 )^{1/2} \right\}.
\end{eqnarray*}
Furthermore, set
\[
H_{b,k}=\bigcup_{h=1}^{T_b(k)} ~ 
\bigcup_{a=0}^{b^{h}-1} ~ 
\bigcup_{n_k^{(b)}/2 \leq \ell \leq n_k^{(b)}} ~ \bigcup_{m_\ell} ~
H(b,k,a, h, \ell, m_\ell).
\]
where the last union is over those $m_\ell \geq 0$ satisfying \eqref{mell}.

The following lemma gives an upper bound for the measure of the set $H_{b,k}$.
 The proof of the lemma will be given in Section~\ref{sec:prooflemmah} below.
\begin{lemma} \label{lemma:measureH}
For $k \geq 100$ and $2 \leq b \leq k$ we have
\[
\frac{\mu(H_{b,k})}{\mu(\Omega_{k-1})} \leq \frac{1}{2^{b}}.
\]
\end{lemma}

As in the proof of Lemma \ref{lemma:bound}, 
let $\mathbf{1}_{[a b^{-h},(a+1)b^{-h})} (x)$ denote the indicator
 function of the interval $[a b^{-h},(a+1)b^{-h})$. 
For the function $F$ appearing in the definition of $H(b,k,a, h, \ell, m_\ell) $, we can write
\begin{eqnarray}
&& F \left(N_k^{(b)} + 4k + m2^\ell,2^{\ell-1},ab^{-h},(a+1)b^{-h},(\{b^j x\})_{j \geq 0} \right) \label{funct_f} \\
& = & \left| \sum_{j=N_k^{(b)} + 4k + m2^\ell}^{N_k^{(b)} + 4k + m2^\ell+2^{\ell-1} -1} \left( \mathbf{1}_{[a b^{-h},(a+1)b^{-h})} (\{b^j x\}) - b^{-h} \right) \right|. \nonumber
\end{eqnarray}
Note that the function $\mathbf{1}_{[a b^{-h},(a+1)b^{-h})} (x)$ is a step
 function which is constant on intervals ranging from one integer multiple 
of $b^{-h}$ to the next (it is zero everywhere, except from $ab^{-h}$ to $(a+1)b^{-h}$, 
where it is one). Accordingly, for some~$j$, the function
\[
\mathbf{1}_{[a b^{-h},(a+1)b^{-h})} (\{b^j x\})
\]
is a step function which is constant on intervals ranging from one integer 
multiple of $b^{-h} b^{-j}$ to the next. Thus the function in line~\eqref{funct_f} 
is constant on all intervals ranging from one integer multiple of 
$b^{-h} b^{-(N_k^{(b)} + 4k + m2^\ell+2^{\ell-1} -1)}$ to the next, and thus 
by $h \leq T_b(k)$ and by~\eqref{mell} it is also constant on all intervals ranging 
from one integer multiple of $b^{-T_b(k)} b^{-N_{k+1}^{(b)}}$ to the~next.

As a consequence, the set $H_{b,k}$ consists of intervals whose
 left and right endpoints are integer multiples of 
\begin{equation} \label{hbkform}
b^{-T_b(k)} b^{-N_{k+1}^{(b)}} =  b^{- \left\lceil \frac{n_k^{(b)} \log 2}{2 \log b} \right\rceil} b^{-\left \lceil \frac{2^{k+1} \log 2}{\log b}  \right\rceil}.
\end{equation}
We call these intervals ``elementary intervals''. We have
\[
b^{- \left\lceil \frac{n_k^{(b)} \log 2}{2 \log b} \right\rceil} \geq 2^{-n_k^{(b)}/2} b^{-1} \geq 2^{-(\log_2 N_{k+1}^{(b)})/2 - 1} b^{-1} \geq 2^{-k/2-2} b^{-1},
\]
and
\[
b^{-\left \lceil \frac{2^{k+1} \log 2}{\log b}  \right\rceil} \geq 2^{-2^{k+1}} b^{-1},
\]
So the length of these elementary intervals of $H_{b,k}$ 
is at least $2^{-2^{k+1}-k/2-2} b^{-2}$.

Let $H_{b,k}^*$ denote the collection of all those intervals of the form 
\begin{equation} \label{hbk:dyad}
\left[a 2^{-2^{k+1} - k}, (a+1) 2^{-2^{k+1} - k} \right) \qquad \text{for some integer $a$}
\end{equation}
which have non-empty intersection with $H_{b,k}$. Note that by the calculations in the previous paragraph the intervals of the form \eqref{hbk:dyad} are much shorter than the elementary intervals of $H_{b,k}$, and thus the total measure of $H_{b,k}^*$ is just a little bit larger than that of $H_{b,k}$. In particular, it is true that
\[
\mu (H_{b,k}^*) \leq \frac{11}{10} \mu (H_{b,k}).
\]
Consequently, by Lemma~\ref{lemma:measureH} we have
\[
\mu \left( \Omega_{k-1} \backslash \bigcup_{b=2}^{k} H_{b,k}^* \right) \geq \left(1 - \frac{11}{10} \sum_{b=2}^{k} \frac{1}{2^b}\right) \mu(\Omega_{k-1}) \geq \frac{9}{20} \mu(\Omega_{k-1}).
\]
Thus, there exists an interval of the form \eqref{hbk:dyad} which is contained 
in $\Omega_{k-1}$, but has empty intersection with all the sets $H_{b,k}$ 
for $b=2, \dots, k$. We define $\Omega_k$ as this interval, and note that the length of $\Omega_k$ is
\begin{equation} \label{omega_k_measure_2}
\mu(\Omega_k) = 2^{-2^{k+1} - k}.
\end{equation}
Now we can make the induction step $k \mapsto k+1$, 
where \eqref{omega_k_measure_2} guarantees that the induction 
hypothesis \eqref{omega_measure} is met.

\section{Proof of Lemma~\ref{lemma:measureH}} \label{sec:prooflemmah}

We use Corollary~\ref{lemma:bound:cor}  to estimate the measure of 
the sets $H(b,k,a, h, \ell, m_\ell)$. More precisely, we apply
the corollary with the choice of 
\[
j_0 = N_k^{(b)} + 4k, \qquad M = m_\ell 2^\ell, \qquad N = 2^{\ell-1}, \qquad A = \Omega_{k-1}, \qquad \varepsilon = 46 (n_k^{(b)}-\ell+1)^{1/2} h^{-2},
\]
where 
\[
1\leq h\leq T_b(k), \qquad n_k^{(b)}/2\leq \ell\leq n_k^{(b)},
\]
and $m_\ell$ satisfies \eqref{mell}. 
So, 
\[
\varepsilon \sqrt{hN} = 46\cdot 2^{(\ell-1)/2} h^{-3/2} (n_k^{(b)} -\ell+1 )^{1/2}.
\]
For the corollary to be applicable, we have to check whether 
$N \geq h$ and \eqref{eps:criterion} hold for our choice of variables. 
However, both conditions are easily seen to be satisfied, since by assumption we have 
$N \geq 2^{n_k^{(b)}/2-1}$, which depends on $k$ exponentially, while 
$h \leq T_b(k) \leq \left\lceil \frac{n_k^{(b)} \log 2}{2 \log b} \right\rceil$ 
and $\varepsilon \leq 46 \sqrt{n_k^{(b)}}$ grow in $k$ at most linearly 
(remember that we assumed $k \geq 100$).
Thus, we can apply Corollary~\ref{lemma:bound:cor}, and we obtain
\begin{eqnarray*}
\mu(H(b,k,a, h, \ell, m_\ell)) & \leq & \mu(\Omega_{k-1}) 2 h  \exp\left(-\frac{46^2 b (n_k^{(b)}-\ell+1) h^{-4} h^5}{529} \right) + 2b^{-j_0} \\
& = & \mu(\Omega_{k-1}) 2 h  \exp\left(-4 b (n_k^{(b)}-\ell+1) h \right) + 2b^{-j_0}.
\end{eqnarray*}
Note that by \eqref{omega_measure},
\[
2b^{-j_0} = 2 b^{-N_k^{(b)}-4k} \leq 2 b^{-\frac{2^{k} \log 2}{\log b}+1-4k} \leq 2b2^{-2^k}b^{-4k} \leq 2b^{-3k+1} \mu(\Omega_{k-1}).
\]
Using the facts that $T_b(h) \leq k/2+1$, 
that $n_{k}^{(b)} \leq k$ for all $b$, and that \eqref{mell} 
implies that there are at most $2^{n_k^{(b)}-\ell} \leq 2^k$ 
different values for $m_\ell$, we obtain
\begin{eqnarray*}
& & \sum_{h=1}^{T_b(k)} ~\sum_{a=0}^{b^{h}-1} ~\sum_{n_k^{(b)}/2 \leq \ell \leq n_k^{(b)}} ~\sum_{m_\ell} 2b^{-3k+1} \\
& \leq & (k/2+1) b^{k/2 +1} k 2^k 2b^{-3k+1} \\
& \leq & \frac{1}{10} b^{-k},
\end{eqnarray*}
where for the last inequality we use the fact that $k \geq 100$ (by assumption).

Furthermore, using the fact that $e^{-xy} \leq e^{-x} e^{-y}$ 
for $x,y \geq 2$, we have
\begin{eqnarray}
& & \sum_{h=1}^{T_b(k)} ~\sum_{a=0}^{b^h-1} ~ \sum_{n_k^{(b)}/2 
\leq       \ell \leq n_k^{(b)}} ~\sum_{m_\ell} ~2 h  \exp\left(-4 bh (n_k^{(b)}-\ell+1)  \right) \nonumber\\
& \leq & \sum_{h=1}^{T_b(k)} b^h ~\sum_{n_k^{(b)}/2 \leq \ell \leq n_k^{(b)}}  2^{n_k^{(b)}-\ell} ~2 h \ \exp({-2bh})\ \exp\big({-2 \big(n_k^{(b)}-\ell+1\big)}\big) \nonumber\\
& \leq & \underbrace{\left( \sum_{h=1}^{\infty} 2h b^h \exp({-2bh}) \right)}_{\leq 11/10 e^{-b}} \underbrace{\sum_{n_k^{(b)}/2 \leq \ell \leq n_k^{(b)}} \exp\big({-\big(n_k^{(b)}-\ell+1\big)}\big)}_{\leq \sum_{r=1}^\infty e^{-r} \leq 6/10} \nonumber\\
& \leq & \frac{7}{10} e^{-b}, \nonumber
\end{eqnarray}
where we used that $b - \log b \geq 1.3$ for $b \geq 2$ and consequently
\[
\sum_{h=1}^{\infty} 2h b^h e^{-2bh} = 
\sum_{h=1}^\infty 2 h e^{- h (b -\log b)} e^{-bh} \leq e^{-b} \underbrace{\sum_{h=1}^\infty 2 h e^{- 1.3h}}_{\leq 11/10} \leq \frac{11}{10} e^{-b}.
\]
Thus, we have
\begin{eqnarray*}
\mu (H_{b,k}) & \leq & \frac{7}{10} e^{-b} \mu(\Omega_{k-1}) + \frac{1}{10} b^{-k} \mu(\Omega_{k-1}) \\
& \leq & 2^{-b} \mu(\Omega_{k-1}),
\end{eqnarray*}
where we used the assumption that $b \leq k$. This proves the lemma.

\section{Proof of Theorem~\ref{th:main}} 

The proof of Theorem~\ref{th:main} now follows using well-known arguments, 
which allow to turn the estimates for subsums over dyadic subsets of the index 
set and over dyadic subintervals of the unit interval into a result which holds 
uniformly over all subintervals in the unit interval, and for all initial segments 
of the full index set.
 
Let $b \geq 2$ be given, and assume that $N$ is ``large'' (depending on $b$). 
Then there is a number $k$ such that $N$ is contained in either 
$\mathcal{N}_k^b$ or $\mathcal{R}_k^b$. 
Let $x$ be a real number which is contained in $\bigcap_{j \geq 1} \Omega_j$. 
Such a number exists, since $(\Omega_j)_{j\geq 1}$ is a sequence of non-empty nested intervals. 
Then for this $x$ we have, for arbitrary $0 \leq \alpha_1 < \alpha_2 \leq 1$,

\begin{eqnarray}
\MoveEqLeft F(0,N,\alpha_1, \alpha_2, \{b^j x\}_{j\geq 1}) &\leq  \nonumber\\
& & F(0,N_{\lfloor k/2 \rfloor }^{(b)},\alpha_1,\alpha_2,(\{b^j x\})_{j \geq 0}) \label{line:1}\\
&&+\ \sum_{r=\lfloor k/2 \rfloor}^{k-1} F(N_r^{(b)}+4r, N_{r+1}^{(b)} - (N_r^{(b)} + 4r) , \alpha_1, \alpha_2, \{b^j x\}_{j \geq 0}) \label{line:2}\\
&&+\ F(N_{k}^{(b)},N-N_{k}^{(b)}-4k, \alpha_1, \alpha_2, \{b^j x\}_{j \geq 0}) \label{line:3}\\
&&+\ \# \left\{j:~j \in \bigcup_k \mathcal{R}_k^b,~j \leq N \right\}. \label{line:4}
\end{eqnarray}

The term in line \eqref{line:1} is bounded by 
$N_{\lfloor k/2 \rfloor }^{(b)} \leq 2^{k/2}{\frac{ \log 2}{\log b}} +1 \leq 2\sqrt{N}$, 
since $N >N_k^{(b)} \geq 2^k{\frac{ \log 2}{\log b}}$ by assumption. 
%
%
Now we bound the term in line \eqref{line:2}.
By Remark~\ref{rk:basic} and Lemma~\ref{lemma:philipp}
and using the definition of the sets $H_{k,b}$, 
for every $r$ such that $\lfloor k/2\rfloor\leq r \leq k-1$, we have
\begin{eqnarray*}
&&  F(N_r^{(b)}+4r, N_{r+1}^{(b)} - (N_r^{(b)} + 4r) , \alpha_1, \alpha_2, \{b^j x\}_{j \geq 0}) \\
&\leq & \sqrt{N_{r+1}^{(b)}} + 
2 (b-1) \sum_{n^{(b)}_{r}/2 \leq \ell \leq n_r^{(b)}} 
\sum_{h=1}^\infty  \max_{ 0\leq a < b^h} F(N_r^{(b)}+4r + m_\ell 2^\ell, 2^{\ell-1}, a b^{-h}, (a+1)b^{-h}, \{b^j x\}_{j \geq 0})
\\&& \qquad \qquad \qquad \qquad \qquad\qquad \qquad \qquad 
\text{for integers $m_1,m_2, \ldots m_{n_r^{(b)}}$ established in Lemma \ref{lemma:philipp},}
\\
& \leq & \sqrt{N_{r+1}^{(b)}} + 2(b-1) ~\sum_{n^{(b)}_{r}/2 \leq \ell \leq n_r^{(b)}} ~\sum_{h=1}^\infty h^{-3/2} 46 \cdot 2^{(\ell-1)/2} (n_r^{(b)}-\ell+1)^{1/2} \\
& \leq & \sqrt{N_{r+1}^{(b)}} + 2(b-1)\cdot 46 \underbrace{\left(\sum_{h=1}^\infty h^{-3/2} \right)}_{\leq 2.62} \underbrace{\left(\sum_{n_r^{(b)}/2 \leq \ell \leq n_r^{(b)}} 2^{(\ell-1)/2} (n_r^{(b)}-\ell+1)^{1/2}\right)}_{\leq 2^{n_r^{(b)}/2} \sum_{u=1}^\infty 2^{-u/2} u^{1/2} \leq 4.15 \cdot 2^{n_r^{(b)}/2} \leq 2.94 \sqrt{N_{r+1}^{(b)}}}  \nonumber\\
& \leq & \sqrt{N_{r+1}^{(b)}} + 709\cdot  b \sqrt{N_{r+1}^{(b)}} \\
& \leq & 710 \cdot b \cdot 2^{(r+1-k)/2} \sqrt{N_{k}^{(b)}}.
\end{eqnarray*}
Consequently, for the term in line \eqref{line:2} we get
\begin{eqnarray*}
\sum_{r=\lfloor k/2 \rfloor}^{k-1} 
F(N_r^{(b)}+4r, N_{r+1}^{(b)} - (N_r^{(b)} + 4r) , \alpha_1, \alpha_2, \{b^j x\}_{j \geq 0})
 & \leq &711\cdot b \sum_{r=\lfloor k/2 \rfloor}^{k-1} 2^{(r+1-k)/2} \sqrt{N_{k}^{(b)}} \\
& \leq & 2425\cdot b \sqrt{N_{k}^{(b)}} \\
& \leq  &2425 \cdot b \sqrt{N}.
\end{eqnarray*}

Similarly, for the term in line \eqref{line:3} we get
\begin{eqnarray*}
F(N_{k}^{(b)},N-N_k^{(b)}-4k, \alpha_1, \alpha_2, \{b^j x\}_{j \geq 0}) & \leq & 711\cdot b\ \sqrt{N_{k+1}^{(b)}} \\
& \leq & 1005 \cdot b \sqrt{N}, 
\end{eqnarray*}
where we used the fact that $2N \geq N_{k+1}^{(b)}$. 

Finally, the term in line \eqref{line:4} is bounded by $4k \lceil k/2 \rceil \leq \sqrt{N}$ 
for sufficiently large $N$. 

Concluding our estimates for the lines \eqref{line:1}--\eqref{line:4}, we finally get
\[
F(0,N,\alpha_1, \alpha_2, \{b^j x\}_{j\geq 0})  \leq
(2 + 2425 + 1005 + 1) \ b \sqrt{N} = 3433\cdot b \sqrt{N},
\]
for all sufficiently large $N$. This can be written in the form
\[
D_N((\{b^j x\})_{j \geq 0}) \leq 3433\cdot  b\ N^{-1/2}
\]
for sufficiently large $N$, which proves the theorem.

\subsection{Computational complexity}

 The real number determined by our construction is the unique
element $x$ in $\bigcap_{k\geq  1} \Omega_k$. 
The definition  of  $(\Omega_k)_{k\geq 1}$ is inductive.
Assume that $\Omega_{k-1}$ is given. 
The interval $\Omega_{k}$  is the leftmost interval 
of the form 
\[
\left[ a 2^{-2^{k+1}-k}, (a+1) 2^{-2^{k+1}-k}\right)
\]
that lies in 
\[
\Omega_{k-1}\ \setminus \ \bigcup_{b=2}^{k} H_{b,k}.
\]
So, $\Omega_k$ is one element  of the partition of 
 $\Omega_{k-1}$ in $2^{2^{k}+1}$ subintervals.
Let us count first how many mathematical operations 
suffice to test if one of these subintervals is outside $\bigcup_{b=2}^{k} H_{b,k}$.
Recall that 
\[
\bigcup_{b=2}^{k} H_{b,k}=\bigcup_{b=2}^{k} ~\bigcup_{h=1}^{T_b(k)} ~ 
\bigcup_{a=0}^{b^{h}-1} ~ 
\bigcup_{n_k^{(b)}/2 \leq \ell \leq n_k^{(b)}} ~ \bigcup_{m_\ell} ~
H(b,k,a, h, \ell, m_\ell).
\]
where the last union is over those $m_\ell \geq 0$ satisfying 
$N_k^{(b)} + 4k + m_\ell 2^\ell + 2^{\ell-1} \leq N_{k+1}^{(b)}$.
 There are at most $ 2^{n^{(b)}_k-\ell}$ values of $m_\ell$, and for each of them 
the evaluation of 
\[
F(N_k^{(b)} + 4k + m2^{\ell},2^{\ell-1},a b^{-h},(a+1) b^{-h},\{b^j x\}_{j\geq 0}),
\]
for   any  fixed $x$ and fixed $k,b,a,h$,
requires  the inspection at~$2^{\ell-1} $ indices.
Hence, the total number of  indices involved in the  inspection is 
$ 2^{n^{(b)}_k-\ell} \cdot 2^{\ell-1} = 2^{n^{(b)}_k}$.
Observe that  $b\leq k$, $T_b(k)<k/2$, $n_k^{(b)}\leq k$, 
and that there are $k/2$ values of $\ell$.
Thus, 
ignoring the operations needed to perform base  change,
the   number of mathematical operations 
to test if  one 
candidate subinterval is  outside $\bigcup_{b=2}^{k} H_{b,k}$ is at most
\[
k \cdot k/2 \cdot k^{k/2}  2^{k}.
\]
In the worst case we need to test all the subintervals (except the last one), and there  are
\[
2^{2^k +1}
\]
of them. This last factor dominates the  total number of mathematical operations
that should be performed in the worst case at step $k$, which consequently is the order at most $\mathcal{O}\left(2^{2^{k+1}}\right)$, say.
Since this is doubly exponential in $k$, the number of mathematical operations performed 
from step $1$ up  to step $k$ is also at most of order~$\mathcal{O}\left(2^{2^{k+1}}\right)$.

At step $k$  the construction determines  $2^{k}+1$ new digits in the binary expansion of 
 the  defined number $x$. 
Thus, at the end of  step $k$   the first $2^{k+1}$ digits will be determined.
Then, to compute the $N$-th digit in the binary expansion of $x$ 
it suffices to  compute up to step 
$\lceil\log_2 N\rceil$. This  entails a number of mathematical operations that is at most 
$\mathcal{O}\left(2^{2^{\log N}}\right) =\mathcal{O}\left(2^N\right) $.
This proves that  there is an algorithm that computes the first 
$N$ digits of the binary expansion of $x$
after performing a number of operations that is exponential in $N$.

\subsection*{Acknowledgements}
We thank Yann Bugeaud and Katusi Fukuyama for comments concerning this paper. The first author is supported by the Austrian Science Fund FWF, project Y-901. The second author is  supported by ANPCyT project PICT 2014-3260.  
The third author is supported by FWF projects
I 1751-N26; W1230, Doctoral Program ``Discrete Mathematics''
and SFB F 5510-N26.  
The fourth author is partially supported by the
National Science Foundation grant DMS-1600441. 
This work was initiated during a workshop on normal numbers at the Erwin Schr\"odinger Institute in Vienna, Austria, which was held in November 2016 and in which all four authors participated.


\bibliographystyle{plain}
\bibliography{ed}

\end{document}